\documentclass[12pt,a4paper]{article}

\usepackage{amsmath,amscd,amssymb,amsthm,mathrsfs,amsfonts,layout,indentfirst,graphicx,caption,mathabx,   stmaryrd,appendix,calc,imakeidx,upgreek}
\usepackage{palatino}  
\usepackage{slashed} 
\usepackage{mathrsfs} 
\usepackage{extarrows} 
\makeindex

\usepackage[nottoc]{tocbibind}   

\usepackage{lipsum}
\let\OLDthebibliography\thebibliography
\renewcommand\thebibliography[1]{
	\OLDthebibliography{#1}
	\setlength{\parskip}{0pt}
	\setlength{\itemsep}{2pt} 
}

\allowdisplaybreaks  
\usepackage{latexsym}
\usepackage{chngcntr}
\usepackage[colorlinks,linkcolor=blue,anchorcolor=blue,linktocpage]{hyperref}
\hypersetup{citecolor=[rgb]{0,0.5,0}}

\usepackage{fullpage}
\counterwithin{figure}{section}

\theoremstyle{definition}
\newtheorem{df}{Definition}[section]

\newtheorem{rem}[df]{Remark}

\theoremstyle{plain}
\newtheorem{thm}[df]{Theorem}

\newtheorem{pp}[df]{Proposition}

\newtheorem{lm}[df]{Lemma}

\newcommand{\fk}{\mathfrak}
\newcommand{\mc}{\mathcal}
\newcommand{\wtd}{\widetilde}
\newcommand{\wht}{\widehat}

\newcommand{\ovl}{\overline}

\newcommand{\End}{\mathrm{End}} 
\newcommand{\id}{\mathbf{1}}
\newcommand{\Hom}{\mathrm{Hom}}

\newcommand{\Vir}{\mathrm{Vir}}

\newcommand{\bk}[1]{\langle {#1}\rangle}

\newcommand{\gk}{\mathfrak g}

\newcommand{\mbb}{\mathbb}
\newcommand{\blt}{\bullet}

\newcommand{\Nbb}{\mathbb N}
\newcommand{\Zbb}{\mathbb Z}

\numberwithin{equation}{section}

\title{Polynomial energy bounds for type $F_4$ WZW-models}
\author{{\sc Bin Gui}
}
\date{}
\begin{document}\sloppy 
	\pagenumbering{arabic}
	\setcounter{section}{-1}

	\maketitle
	
\begin{abstract}
We prove that   for any type $F_4$ unitary affine VOA $V^l_{\fk f_4}$, sufficiently many intertwining operators satisfy polynomial energy bounds. This finishes the Wassermann type analysis of intertwining operators for all WZW-models.
\end{abstract}	
	
\tableofcontents


\section{Introduction}

In his breakthrough work \cite{Was98}, A. Wassermann computed the fusion rules for type $A$ WZW conformal nets, thus establishing the first important relation between the tensor structures of VOA modules and conformal net ones. One of the key steps in \cite{Was98} is to  prove that sufficiently many intertwining operators have bounded smeared primary (i.e. lowest weight) fields. The boundedness condition does not hold in general, and should be replaced by the \emph{(polynomial) energy bounds condition} \cite{CKLW18}, which says roughly that the smeared intertwining operator is bounded by some $n$-th power of the energy operator $L_0$. By generalizing the ideas of Wassermann,  Toledano-Laredo established in \cite{TL04} the energy bounds condition for sufficiently many intertwining operators of type $D$ unitary affine VOAs.  Similar results for type $BCG$ were proved recently by the author in  \cite{Gui19b}. The main strategies are the same (which we call the compression method): One first show the energy bounds for level $1$ intertwining operators. The higher level cases are treated inductively by considering the diagonal embedding $V_\gk^{l+1}\subset V_\gk^1\otimes V_\gk^l$ and showing that sufficiently many $V_\gk^{l+1}$-intertwining operators are compressions (i.e. restrictions) of $V_\gk^1\otimes V_\gk^l$-intertwining operators. Since the latter are energy-bounded by induction, the former are so.

The most difficult and technical part of this  method is to verify that compressing the intertwining operators of the larger VOA (say $V_\gk^1\otimes V_\gk^l$) produce enough intertwining operators of the smaller VOA (say $V_\gk^{l+1}$). In \cite{Was98,TL04,Gui19b}  this is  mostly done  by Lie algebraic methods, which require case by case studies. Generalizing such analysis to type $E$ (especially $E_8$) WZW-models will be still possible but  significantly more difficult. Fortunately, when $\gk$ is of type $ADE$, due to the facts (cf. \cite{ACL19,Ara15a,Ara15b}) that $V_\gk^{l+1}$ and its commutant in $V_\gk^1\otimes V_\gk^l$ are both regular, and that $V_\gk^{l+1}$ equals its double commutant, it is always possible to produce enough intertwining operators via compressions (cf. \cite{KM15,Gui20b}). This proves the energy bounds condition for \emph{all} WZW-intertwining operators  of type $ADE$ (which, even in the case of type $AD$, is stronger than the results in \cite{Was98,TL04}). Moreover, this method is vertex algebraic and does not require case by case studies. Unfortunately, it is not known if the main results in \cite{ACL19} can be applied to other Lie types or not. Thus, so far,  for the remaining type $F_4$, one can only prove the energy bounds using Lie theoretic methods. This is the goal of our present paper. 

Let $\fk f_4$ be the type $F_4$ complex simple Lie algebra. Let $l\in\mbb Z_+$. The main result of this article is that sufficiently many intertwining operators of the unitary affine VOA $V_{\fk f_4}^l$ is energy-bounded. To be more precise, let $\lambda_4$ be the highest weight associated to the smallest non-trivial irreducible $\fk f_4$-module $L_{\fk f_4}(\lambda_4)$ (which is of dimension $26$ and admissible at level $1$), and $L_{\fk f_4}(\lambda_4,l)$ is the irreducible representation of the affine VOA $V_{\fk f_4}^l$ with highest weight $\lambda_4$. We show that any intertwining operator with charge space $L_{\fk f_4}(\lambda_4,l)$ is energy-bounded (Thm. \ref{lb2}), and that any irreducible $V_{\fk f_4}^l$-module is a submodule of a tensor (fusion) product of $L_{\fk f_4}(\lambda_4,l)$ (Thm. \ref{lb22}). (The second result says that considering only intertwining operators with charge space $L_{\fk f_4}(\lambda_4,l)$ is ``sufficient".) Theorem \ref{lb22} is an easy consequence of $V^l_{\fk f_4}$-fusion rules  well known in the literature; theorem \ref{lb2} is the non-trivial part of our main result. In the following, we briefly explain the strategies of the proof by comparing $\fk f_4$ with $\fk g_2$ studied in \cite{Gui19b}.

One might guess that $\fk f_4$ and $\fk g_2$ can be treated in a similar way due to the conformal embedding $V_{\fk f_4}^1\otimes V_{\fk g_2}^1\subset V_{\fk e_8}^1$. However, it turns out that $\fk f_4$ is very different from $\fk g_2$ and from all the classical Lie types in the following two aspects.
\begin{enumerate}
	\item For any complex simple Lie algebra $\gk$ not of type $F_4$ or $E_8$, the weight multiplicities of the smallest non-trivial irreducibles are bounded by $1$. By contrast, $L_{\fk f_4}(\lambda_4)$ has weight $0$ with multiplicity $2$. Consequently, the tensor product rules of $L_{\fk f_4}(\lambda_4)$ exceed $1$, which makes the analysis of $\fk f_4$ more subtle.
	\item It seems very difficult (if not impossible) to reduce the higher level cases to the level $1$ case using the Lie algebraic methods as in \cite{Was98,TL04,Gui19b}. Thus, one also needs to treat level $2$ separately. However, the method  for $\fk g_2$ level $1$ (see \cite{Gui19b} especially lemma 5.5) can be applied only to $\fk f_4$ level $1$ but not level $2$.
\end{enumerate}

We resolve the second issue by exploiting the conformal embeddings
\begin{gather*}
V^1_{\fk {sl}_2}\otimes V^1_{\fk {sp}_6}\subset V^1_{\fk f_4}\\
\Vir^{c_9}\otimes V^2_{\fk {sl}_2}\otimes V^2_{\fk {sp}_6}\subset V^2_{\fk f_4}
\end{gather*} 
where $\Vir^{c_9}$ is the (regular) unitary Virasoro VOA with central charge $c_9<1$. Since the intertwining operators of the smaller VOAs are energy-bounded, so are those of the larger ones. See section \ref{lb17} for more details. As for the first issue, we show that the $2$-dimensional weight $0$ subspace $L_{\fk f_4}(\lambda_4)[0]$ is spanned by two particular vectors (called $F_{\rho_3}v_{\rho_3},F_{\rho_4}v_{\rho_4}$) killed by the homomorphisms in $\Hom_\gk(\lambda_4\otimes\lambda_4,\lambda_4)$  and in $\Hom_\gk(\lambda_4\otimes\lambda_3,\lambda_3)$ respectively. (Here $\lambda_3$ is the highest weight for the irreducible $273$-dimensional representation.) The key result is lemma \ref{lb1}, for which we give two different proofs. The importance of this result is that any homomorphism of the form $\Hom_\gk(\lambda_4\otimes\lambda,\lambda)$ (where $\lambda$ is a dominant integral weight of $\fk f_4$) reduces to a linear combination of those in $\Hom_\gk(\lambda_4\otimes\lambda_4,\lambda_4)$  and  $\Hom_\gk(\lambda_4\otimes\lambda_3,\lambda_3)$. (For instance, see the proof of proposition \ref{lb18}.) 

Section \ref{lb19} contains the most technical part of this article, which reduces the four types of intertwining operators in proposition \ref{lb13} to those of level $1$ and $2$. The main idea is the same as case (II) of \cite{Gui19b} section 5.2. In particular, proposition \ref{lb9} and lemma \ref{lb12}, on which the proof of  proposition \ref{lb13} relies, also appears either explicitly or implicitly in \cite{Gui19b}. Indeed, proposition \ref{lb9} is a (slight) generalization of \cite{Gui19b} lemma 2.15, whereas lemma \ref{lb12} generalizes the arguments in the proof of \cite{Gui19b} lemma 5.6.

The main result of this article, together with those in \cite{Was98,TL04,Gui19b,Gui20b}, is crucial for showing the equivalence of the VOA tensor categories and the conformal net tensor categories associated to WZW-models and their regular cosets. This is the main topic of \cite{Gui20a} and will not be discussed in the present article.

\subsubsection*{Acknowledgment}

The author would like to thank James Tener for helpful discussions.

\section{Energy-bounded intertwining operators}

We refer the readers to \cite{DL14,CKLW18,Gui19a} for the  basics  of unitary VOAs, unitary modules, and unitary intertwining operators. Let $V$ be a unitary VOA. Thus $V$ is equipped with an inner product $\bk{\cdot|\cdot}$ (antilinear on the second variable) and an anilinear automorphism $\Theta$ (preserving the vacuum vector $\Omega$ and the conformal vector $\nu$) such that for every $v,v_1,v_2\in V$,
\begin{align*}
\bk{v_1|Y(v,\ovl z)v_2}=\bk{Y(e^{zL_1}(-z^{-2})^{L_0}\Theta v,z^{-1})v_1|v_2}.
\end{align*}
Moreover,  $\Theta$ is anti-unitary, and $\Theta^2=\id_V$. A $V$-module $(W_i,Y_i)$ is called unitary if $W_i$ is equipped with an inner product $\bk{\cdot|\cdot}$ under which the above relation holds for any $v\in V,v_1,v_2\in W_i$. In particular, $V$ is a unitary $V$-module. The eigenvalues of $L_0$ on $W_i$ are non-negative.

For each $v\in V$ we write $Y(v,z)=\sum_{n\in\mbb Z}Y(v)_nz^{-n-1}$. Recall that if $W_i,W_j,W_k$ are unitary $V$-modules, a type ${W_k\choose W_iW_j}={k\choose i~j}$ (unitary) intertwining operator $\mc Y$  is a linear map
\begin{gather*}
W_i\rightarrow \End(W_j,W_k)\{z\}\\
w_i\mapsto \mc Y(w^{(i)},z)=\sum_{n\in\mbb R}\mc Y(w^{(i)})_nz^{-n-1}
\end{gather*}
where the sum above is the formal sum,  each $\mc Y(w^{(i)})_n$ is in $\End(W_j,W_k)$, and the following conditions are satisfied:

(a) (Lower truncation) For any $w^{(j)}\in W_j$, $\mathcal Y(w^{(i)})_nw^{(j)}=0$ when $n$ is sufficiently large.

(b) (Jacobi identity) For any $u\in V,w^{(i)}\in W_i,m,n\in\mathbb Z,s\in\mbb R$, we have
\begin{align}
&\sum_{l\in\Nbb}{m\choose l}\mathcal Y\big(Y_i(u)_{n+l}w^{(i)}\big)_{m+s-l}\nonumber\\
=&\sum_{l\in\Nbb}(-1)^l{n\choose l}Y_k(u)_{m+n-l}\mathcal Y(w^{(i)})_{s+l}-\sum_{l\in\Nbb}(-1)^{l+n}{n\choose l}\mathcal Y(w^{(i)})_{n+s-l}Y_j(u)_{m+l}.
\end{align}

(c)	($L_{-1}$-derivative) $	\frac d{dz} \mathcal Y(w^{(i)},z)=\mathcal Y(L_{-1}w^{(i)},z)$.\\
We say that $W_i,W_j,W_k$ are respectively the \textbf{charge space}, the \textbf{source space}, and the \textbf{target space} of $\mc Y$. If $W_i,W_j,W_k$ are all irreducible, we say that $\mc Y$ is an \textbf{irreducible intertwining operator}.

Let $w^{(i)}\in W_i$ be a homogeneous vector (for simplicity), i.e. $w^{(i)}$ is an eigenvector of $L_0$. We say that $\mc Y(w^{(i)},z)$ is \textbf{energy-bounded} (or satisfies polynomial energy bounds) if there exist $M,t,r\geq 0$ such that for any $w^{(j)}\in W_j,n\in\mbb R$,
\begin{align*}
\lVert \mc Y(w^{(i)})_nw^{(j)} \lVert\leq M(1+|n|)^t\lVert (1+L_0)^rw^{(j)} \lVert
\end{align*}
where the norms $\lVert \cdot \lVert$ are defined by the inner products of $W_j,W_k$. We say that $\mc Y$ is energy-bounded if $\mc Y(w^{(i)},z)$ is so for any homogeneous $w^{(i)}\in W_i$. We say that $V$ is strongly energy-bounded if $Y_i$ is energy-bounded for any unitary $V$-module $W_i$.

Suppose that $\gk$ is a finite dimensional (unitary) complex simple Lie algebra. Let $(\cdot|\cdot)$ be the (unique) invariant inner product under which the longest roots of $\gk$ have length $\sqrt 2$. For each $l\in\Zbb_+=\{1,2,3,\dots\}$, the \textbf{affine VOA} $V_\gk^l$ is the unique unitary VOA generated by the weight $1$ subspace $V(1)$ such that $V(1)$ (with the naturally defined Lie algebra structure) is equivalent to $\gk$, and that the inner product $\bk{\cdot|\cdot}$ on $V(1)$ equals $l$ times $(\cdot|\cdot)$. $V^l_\gk$ is strongly energy-bounded. (See \cite{CKLW18} example 8.7.) Thus, by \cite{Gui19a} corollary 3.7-(a), we have

\begin{pp}\label{lb16}
Let $\mc Y$ be an irreducible $V^l_\gk$-module with charge space $W_i$. If there exists a homogeneous non-zero $w^{(i)}\in W_i$ such that $\mc Y(w^{(i)},z)$ is energy-bounded, then $\mc Y$ is energy-bounded.
\end{pp}

In the remaining part of this article, unless otherwise stated, we shall let $\gk$ denote the type $F_4$ complex simple Lie algebra $\fk f_4$.

\section{The $26$-dimensional representation $L_\gk(\lambda_4)$}\label{lb6}

We follow the notations in \cite{Gui19b}. Recall that we set $\gk=\fk f_4$. Let $\fk h$ be the Cartan subalgebra of $\gk$. Then $\fk h^*$ is the space of weights of $\gk$. Let $\theta$ be the longest root of $\fk f_4$. Recall that the invariant inner product $(\cdot|\cdot)$ on $\fk f_4$ is chosen such that $(\theta|\theta)=2$. One then have an isomorphism $\fk h\simeq\fk h^*$ induced by this inner product. For any root $\alpha$ and weight $\lambda$, we set
\begin{align}
n_{\lambda,\alpha}=\frac{2(\lambda|\alpha)}{(\alpha|\alpha)}.
\end{align}
$\gk$ has an involution $*$ such that $X^*=-X$ when $X$ is an element in the compact real form. Then, with the $*$-structure, $\gk$ is a unitary Lie algebra. Thus, given a unitary representation $W$ of $\gk$ where $W$ is equipped with an inner product $\bk{\cdot|\cdot}$, we have for any $X\in\gk,u,v\in W$ that
\begin{align*}
\bk{Xu|v}=\bk{u|X^*v}.
\end{align*}
In particular, the adjoint representation $\gk\curvearrowright\gk$ is unitary. Hence, for any $X,Y,Z\in\gk$ we have
\begin{align*}
([X,Y]|Z)=(Y|[X^*,Z]).
\end{align*}

Let $l=\mbb Z_+$. Let $P_+(\gk)$ be the set of dominant integral weights of $\gk$. Let $P_+(\gk,l)$ be the set of all $\lambda\in P_+(\gk)$ admissible at level $l$ (i.e. $(\lambda|\theta)\leq l$). Let $V^l_\gk$ be the level $l$ unitary affine VOA associated to $\gk$. Then the irreducible unitary representations of $V_\gk^l$ are precisely those equivalent to some highest weight (equivalently, lowest energy) representation $L_\gk(\lambda,l)$, where $\lambda\in P_+(\gk,l)$. (See \cite{FZ92}.)  $L_\gk(\lambda)$ denotes the finite dimensional (unitary) irreducible representation of $\gk$ with highest weight $\lambda$. We also identify $L_\gk(\lambda)$ with the lowest energy subspace of $L_\gk(\lambda,l)$. If $\mu\in\fk h$ is a weight, then $L_\gk(\lambda)[\mu]$ the $\mu$-weight space of $\gk$, i.e., $L_\gk(\lambda)[\mu]$ consists of $v\in L_\gk(\lambda)$ such that $hv=(h|\mu)v$ for any $h\in\fk h$. 

$\fk f_4$ has the following simple roots (described by the ``orthogonal basis")
\begin{gather*}
\rho_1=[0,1,-1,0],\qquad \rho_2=[0,0,1,-1],\qquad \rho_3=[0,0,0,1],\qquad \rho_4=\frac 12[1,-1,-1,-1].
\end{gather*}
The corresponding fundamental weights are
\begin{gather*}
\lambda_1=[1,1,0,0],\qquad \lambda_2=[2,1,1,0],\qquad \lambda_3=\frac 12[3,1,1,1],\qquad\lambda_4=[1,0,0,0].
\end{gather*}
We also have $\theta=\lambda_1=[1,1,0,0]$. We also adopt the following notation
\begin{align*}
(n_1,n_2,n_3,n_4)=n_1\lambda_1+n_2\lambda_2+n_3\lambda_4+n_4\lambda_4.
\end{align*}

$\fk f_4$ has $24$ positive roots, whose lengths are either $1$ or $\sqrt 2$. Those with length $1$ fall into the following two groups. The roots in \textbf{Group A} are orthogonal to the highest root $\theta=\lambda_1$: 
\begin{align*}
(0,-1,2,-1)&=[0,0,0,1]\qquad(=\rho_3)\\
(-1,1,0,-1)&=[0,0,1,0]\\
(0,0,-1,2)&=\frac 12[1,-1,-1,-1]\qquad (=\rho_4)\\
(0,-1,1,1)&=\frac 12[1,-1,-1,1]\\
(-1,1,-1,1)&=\frac 12[1,-1,1,-1]\\
(-1,0,1,0)&=\frac 12[1,-1,1,1]
\end{align*}
Group B consists of those not orthogonal to $\theta$. There are also $6$ elements in group B:
\begin{align*}
[0,1,0,0],[1,0,0,0],\frac 12[1,1,\pm1,\pm1]
\end{align*}
Any positive root with length $\sqrt 2$ is of the form $[a,b,c,d]$ where only two of $a,b,c,d$ are non-zero; the first non-zero number is $1$; the second one is $1$ or $-1$. If $\alpha$ is a   root, we choose a raising operator $E_\alpha$. Then $F_\alpha:=E_\alpha^*$ is a lowering operator. We normalize $E_\alpha$ such that $E_\alpha,F_\alpha,H_\alpha=[E_\alpha,F_\alpha]$ satisfy $[H_\alpha,E_\alpha]=2E_\alpha$. Thus $[H_\alpha,F_\alpha]=-2E_\alpha$. Indeed, if $\beta$ is also a root, then $[H_\alpha,E_\beta]=n_{\beta,\alpha}E_\beta,[H_\alpha,F_\beta]=-n_{\beta,\alpha}F_\beta$. More generally, due to the fact that for any $\lambda\in\fk h^*\simeq\fk h$ we have
\begin{align*}
(H_\alpha|\lambda)=n_{\lambda,\alpha},
\end{align*}
if $v$ is a $\lambda$-weight vector, then
\begin{align*}
H_\alpha v=n_{\lambda,\alpha}v.
\end{align*}
We also have
\begin{align*}
(E_\alpha|E_\alpha)=(F_\alpha|F_\alpha)=\frac 2{(\alpha|\alpha)}.
\end{align*}
One can assume furthermore that
\begin{align*}
E_{-\alpha}=F_\alpha,\qquad F_{-\alpha}=E_{\alpha},\qquad H_{-\alpha}=-H_\alpha.
\end{align*}
See \cite{Gui19b} section 1.2 for more details.

$L_\gk(\lambda_4)$ is $26$-dimensional and has $25$ weights: the zero weight,  the twelve roots in group A and group B,  and there negatives. If a weight $\mu$ of $L_\gk(\lambda_4)$ is not zero, then $\dim L_\gk(\lambda_4)[\mu]=1$. On the other hand, $\dim L_\gk(\lambda_4)[0]=2$. These facts can be checked by LieART (see section \ref{lb14}).

The main result of this article is:
\begin{thm}\label{lb2}
Any irreducible intertwining operator of $V^l_{\fk f_4}$ with charge space $L_{\fk f_4}(\lambda_4,l)$ is energy-bounded.
\end{thm}

\section{The case $l=1,2$}\label{lb17}

In this section, we prove theorem \ref{lb2} when $l=1,2$. The proof for level $1$ is similar but slightly simpler than level $2$. So we will mainly focus on level $2$. $\gk=\fk f_4$ has a unitary Lie subalgebras $\fk a_1$ and $\fk c_3$ where $\fk a_1$ is generated by the raising and the lowering operators of $\lambda_1$ and $\fk c_3$ is generated by those of $\rho_2,\rho_3,\rho_4$. Then $\fk a_1$ and $\fk c_3$ are simple Lie algebras of type $A_1,C_3$ respectively, i.e., $\fk a_1=\fk{sl}_2,\fk c_3=\fk{sp}_6$. By the fact that $\rho_2\pm\lambda_1,\rho_3\pm\lambda_1,\rho_2\pm\lambda_1$ are not roots of $\fk f_4$ since there lengths exceed $\sqrt 2$, we have $[\fk a_1,\fk c_3]=0$, which shows that there is an embedding of unitary Lie algebras $\fk a_1\oplus \fk c_3\subset \fk f_4$. Moreover, the long roots of $\fk a_1,\fk c_3$ both have lengths $\sqrt 2$ under the normalized invariant inner product of $\gk$. Thus, the Dynkin indexes of $\fk a_1\subset\fk f_4$ and $\fk c_3\subset\fk f_4$ are both $1$, and one thus have unitary vertex operator subalgebras
\begin{align*}
V^l_{\fk a_1}\otimes V^l_{\fk c_3}\subset V^l_{\fk f_4}.
\end{align*}
(See, for example, \cite{Gui19b} section 2.1.) By comparing the central charges, one sees that when $l=1$, the above subalgebra is a conformal subalgebra, i.e., both sides have the same central charge and hence the same conformal vector. When $l=2$, the difference between the two central charges is $c_9=1-\frac 6{9\cdot 10}$. Thus we have conformal subalgebras
\begin{gather*}
V^1_{\fk a_1}\otimes V^1_{\fk c_3}\subset V^1_{\fk f_4},\\
\Vir^{c_9}\otimes V^2_{\fk a_1}\otimes V^2_{\fk c_3}\subset V^2_{\fk f_4},
\end{gather*} 
where $\Vir^{c_9}$ is the unitary Virasoro VOA with central charge $c_9$. By \cite{DLM97}, unitary affine VOAs and unitary Virasoro VOAs with central charge $c<1$ are regular. So $\Vir^{c_9}\otimes V^2_{\fk a_1}\otimes V^2_{\fk c_3}$ is regular. Thus any unitary $V^2_{\fk f_4}$-module, as a unitary $\Vir^{c_9}\otimes V^2_{\fk a_1}\otimes V^2_{\fk c_3}$-module, has a finite orthogonal irreducible decomposition, and each irreducible component has the form $W_1\otimes W_2\otimes W_3$, where $W_1,W_2,W_3$ are irreducible representations of $\Vir^{c_9},V^2_{\fk a_1},V^2_{\fk c_3}$ respectively. (Cf. \cite{Ten19} proposition 2.20.)

\begin{lm}\label{lb20}
$L_{\fk f_4}(\lambda_4,2)$ has an irreducible unitary $\Vir^{c_9}\otimes V^2_{\fk a_1}\otimes V^2_{\fk c_3}$-submodule
\begin{align}
W_1\otimes L_{\fk a_1}(\Box,2)\otimes L_{\fk c_3}(\vartheta_1,2),\label{eq3}
\end{align}
where $W_1$ is an irreducible unitary $\Vir^{c_9}$-module, $\Box$ is the highest weight of the ($2$-dimensional) vector representation of  $\fk{sl}_2$, and $\vartheta_1$ is the highest weight of the ($6$-dimensional) vector representation of  $\fk{sp}_6$.
\end{lm}

The case of level $1$ is similar and is left to the reader.

\begin{proof}
The lowest energy subspace of $L_{\fk f_4}(\lambda_4,2)$ is $L_{\fk f_4}(\lambda_4)$. Since the restrictions of $\lambda_4$ to the Cartan subalgebras of $\fk a_1$ and $\fk c_3$ equal $\Box$ and $\vartheta_1$ respectively, the $\fk f_4$-module  $L_{\fk f_4}(\lambda_4)$ has an irreducible $(\fk a_1\oplus\fk c_3)$-submodule $L_{\fk a_1}(\Box)\otimes L_{\fk c_3}(\vartheta_1)$. Thus, the weak $V_{\fk a_1}^2\otimes V_{\fk c_3}^2$-submodule of $L_{\fk f_4}(\lambda_4,2)$ generated by $L_{\fk a_1}(\Box)\otimes L_{\fk c_3}(\vartheta_1)$ is equivalent to $L_{\fk a_1}(\Box,2)\otimes L_{\fk c_3}(\vartheta_1,2)$. Recall that $L_{\fk f_4}(\lambda_4,2)$ is a finite sum of irreducible  $\Vir^{c_9}\otimes V^2_{\fk a_1}\otimes V^2_{\fk c_3}$-submodules of the form $W_1\otimes W_2\otimes W_3$ which, as a weak $V_{\fk a_1}^2\otimes V_{\fk c_3}^2$-module, is a direct sum of $W_2\otimes W_3$. Therefore, there must be an irreducible submodule $W_1\otimes W_2\otimes W_3$ with $W_2\otimes W_3\simeq L_{\fk a_1}(\Box,2)\otimes L_{\fk c_3}(\vartheta_1,2)$.
\end{proof}

\begin{pp}\label{lb4}
Theorem \ref{lb2} holds when $l$ equals $1$ or $2$.
\end{pp}
\begin{proof}
We only discuss the case of $l=2$ since the other case can be treated in a similar way. Let $\mc Y$ be any intertwining operator of $V^2_{\fk f_4}$ with charge space $L_{\fk f_4}(\lambda_4,2)$. By proposition \ref{lb16}, it suffices to prove that $\mc Y(w,z)$ is energy-bounded for some non-zero homogeneous vector $w\in L_{\fk f_4}(\lambda_4,2)$. If we regard $\mc Y$ as an intertwining operator of $\Vir^{c_9}\otimes V^2_{\fk a_1}\otimes V^2_{\fk c_3}$, then, by the above lemma, one can restrict the charge space of $\mc Y$ of a charge subspace of the form \eqref{eq3}. Let $\mc Y_0$ be the restriction of $\mc Y$ to this charge subspace. Then it suffices to prove that $\mc Y_0$ is energy-bounded. 

By \cite{ADL05} theorem 2.10, $\mc Y_0$ is a finite sum of  intertwining operators of the form $\mc Y_1\otimes \mc Y_2\otimes \mc Y_3$ where $\mc Y_1,\mc Y_2,\mc Y_3$ are irreducible intertwining operators of $\Vir^{c_9},V^2_{\fk a_1},V^2_{\fk c_3}$ with charge spaces $W_1,L_{\fk a_1}(\Box,2),L_{\fk c_3}(\vartheta_1,2)$  respectively. By \cite{Gui19b} theorem 4.2, $\mc Y_3$ is energy-bounded. By \cite{Was98}, $\mc Y_2$ is energy-bounded. By \cite{Loke94} proposition IV.1.3, any intertwining operator of a unitary $c<1$ Virasoro VOA is energy-bounded.\footnote{This also follows from the fact that any intertwining operator of a type $A$ discrete series $W$-algebra is energy-bounded. See the introduction of \cite{Gui20b}.}  So $\mc Y_1$ is energy-bounded. By the arguments in \cite{CKLW18} section 6 or \cite{Gui19a} proposition 3.5, tensor products of energy-bounded intertwining operators are energy-bounded. So $\mc Y_0$ is energy-bounded.
\end{proof}

\begin{rem}
The method in this section can be used to prove a similar result for the level $1$ affine type $G_2$ VOA $V_{\gk_2}^1$. Let $\alpha_1=[\sqrt {2/3},0],\alpha_2=[-\sqrt {3/2},\sqrt {1/2}]$ be the simple roots of $\fk g_2$. Then $\alpha_3=[0,\sqrt 2]$ is  a root of $\gk$ with squared length $2$.  We have embedding $\fk a_1\oplus\fk a_1\subset\fk g_2$ with the first $\fk a_1$ generated by the raising and the lowering operators of $\alpha_1$, and the second one generated by those of $\alpha_3$. This embedding induces a conformal extension $V_{\fk a_1}^3\otimes V_{\fk a_1}^1\subset V_{\gk_2}^1$. Let $\varsigma=[\sqrt{1/6},\sqrt{1/2}]$ which is the highest weight of the $7$-dimensional irreducible $\gk_2$-module. Then the $V_{\gk_2}^1$-module $L_{\gk_2}(\varsigma,1)$ has an irreducible $V_{\fk a_1}^3\otimes V_{\fk a_1}^1$-submodule $L_{\fk a_1}(\Box,3)\otimes L_{\fk a_1}(\Box,1)$.   Since the intertwining operators of $V_{\fk a_1}^3\otimes V_{\fk a_1}^1$ with such charge space are energy-bounded, so are those of $ V_{\gk_2}^1$ with charge space $L_{\gk_2}(\varsigma,1)$.
\end{rem}

\section{The tensor (fusion) product rule $N^\nu_\mu$}

For each $\lambda\in P_+(\gk,l)$, we let $\Delta_\lambda$ be the conformal weight of $L_\gk(\lambda,l)$, i.e. $\Delta_\lambda$ is the smallest eigenvalue of $L_0$ on $L_\gk(\lambda,l)$. Note that $\Delta_\lambda$ depends on $l$. Indeed, we have $\Delta_\lambda=\frac{C_\lambda}{2(l+h^\vee)}$ where $h^\vee$ is the dual Coxeter number of $\gk$, and $C_\lambda$ is the Casimir number of $L_\gk(\lambda)$.

We denote by $\mc V_\gk^l{\nu\choose\lambda~\mu}$ the vector space of type ${\nu\choose\lambda~\mu}={L_\gk(\nu,l)\choose L_\gk(\lambda,l)~L_\gk(\mu,l)}$ intertwining operators of $V_\gk^l$, assuming that $\lambda,\mu,\nu$ are admissible at level $l$. We let $N^\nu_{\lambda\mu}$ be the dimension of this vector space. Since, in this article, we will be mainly interested in the case that $\lambda=\lambda_4$, we let
\begin{align}
N^\nu_\mu:=N^\nu_{\lambda_4\mu}=\dim\mc V^l_\gk{\nu\choose\lambda_4~\mu}.
\end{align}

Set
\begin{align*}
\Delta^\nu_{\lambda\mu}=\Delta_\lambda+\Delta_\mu-\Delta_\nu.
\end{align*}
We also write $\Hom_\gk(L_\gk(\lambda)\otimes L_\gk(\mu),L_\gk(\nu))$ as $\Hom_\gk(\lambda\otimes\mu,\nu)$ for short. Then for any $\mc Y\in\mc V_\gk^l{\nu\choose\lambda~\mu}$ and $u^{(\lambda)}\in L_\gk(\lambda),v^{(\mu)}\in L_\gk(\mu)$, we have $\mc Y(u^{(\lambda)})_{\Delta^\nu_{\lambda\mu}-1}\cdot u^{(\mu)}\in L_\gk(\nu)$. We define the linear map
\begin{align}\label{eq4}
\Psi:\mc V_\gk^l{\nu\choose\lambda~\mu}\rightarrow\Hom_\gk(\lambda\otimes\mu,\nu)
\end{align}
by sending each element $\mc Y$ to $\Psi\mc Y$ satisfying
\begin{gather*}
\Psi\mc Y(u^{(\lambda)}\otimes u^{(\mu)})=\mc Y(u^{(\lambda)})_{\Delta^\nu_{\lambda\mu}-1}\cdot u^{(\mu)}.
\end{gather*}
It is well known that $\Psi$ is injective. Moreover, if $(\lambda|\theta)=1$ (for example, if $\lambda=\lambda_4$) then $\Psi$ is also surjective. (See \cite{Fuc94} (5.9).) In this case, the fusion rule agrees with the truncated tensor product rule:
\begin{align*}
N^\nu_\lambda=\dim\Hom_\gk(\lambda_4\otimes\mu,\nu).
\end{align*}
In particular, $N^\nu_\mu$ is independent of the level $l$ at which  $\mu,\nu$ are admissible.

Thus, to study the dimension of $\mc V_\gk^l{\nu\choose\lambda_4~\mu}$, it suffices to understand the tensor product rules of $L_\gk(\lambda_4)$. Let $v_\mu\in L_\gk(\mu)[\mu],v_\nu\in  L_\gk(\nu)[\nu]$ be (non-zero) highest weight vectors of $L_\gk(\mu),L_\gk(\nu)$ respectively.  Define a linear map
\begin{gather}
\Gamma:\Hom_\gk(\lambda\otimes\mu,\nu)\rightarrow L_\gk(\lambda)[\nu-\mu]^*,\label{eq5}
\end{gather}
such that for any $T\in\Hom_\gk(\lambda\otimes\mu,\nu)$, $\Gamma T$, as a linear functional on $L_\gk(\lambda)[\nu-\mu]$, is defined by
\begin{align}
(\Gamma T)(u^{(\lambda)})=\bk{T(u^{(\lambda)}\otimes v_\mu)|v_\nu}
\end{align}
for any $u^{(\lambda)}\in L_\gk(\lambda)[\nu-\mu]$. Then $\Gamma$ is injective.

The image of $\Gamma$ can be described as follows. Let $K_\gk^\mu(\lambda)[\nu-\mu]$ be the subspace of $L_\gk(\lambda)[\nu-\mu]$ spanned by vectors of the form $F_\alpha^{n_{\mu,\alpha}+1}u^{(\lambda)}$, where $\alpha$ is a simple root of $\gk$, and $u^{(\lambda)}\in L_\gk(\lambda)$ has weight $\nu-\mu+(n_{\mu,\alpha}+1)\alpha$. Denote by $K_\gk^\mu(\lambda)[\nu-\mu]^\perp$ the set of elements of $L_\gk(\lambda)[\nu-\mu]^*$ vanishing on $K_\gk^\mu(\lambda)[\nu-\mu]$. Then by \cite{Gui19b} proposition 1.17, we have
\begin{align}
\boxed{~~\mathrm{Im}(\Gamma)=K_\gk^\mu(\lambda)[\nu-\mu]^\perp.~~}\label{eq1}
\end{align}
Thus
\begin{align}
N^\nu_{\lambda\mu}=\dim L_\gk(\lambda)[\nu-\mu]-\dim K_\gk^\mu(\lambda)[\nu-\mu].\label{eq2}
\end{align}

For example, fix non-zero vectors
\begin{align}
v_{\rho_3}\in L_\gk(\lambda_4)[\rho_3],\qquad v_{\rho_4}\in L_\gk(\lambda_4)[\rho_4].\label{eq7}
\end{align}
Since we know that a simple root $\alpha$ must be one of $\rho_1,\rho_2,\rho_3,\rho_4$, it is easy to see that
\begin{gather}
K_\gk^{\lambda_3}(\lambda_4)[0]=\mbb C\cdot F_{\rho_4}v_{\rho_4},\qquad K_\gk^{\lambda_4}(\lambda_4)[0]=\mbb C\cdot F_{\rho_3}v_{\rho_3}.
\end{gather}
As we shall see immediately, these two vectors are non-zero and (indeed) linearly independent. Thus we have $N^{\lambda_3}_{\lambda_3}=N^{\lambda_4}_{\lambda_4}=1$. More generally, we have:

\begin{lm}\label{lb5}
Let $\mu=n_1\lambda_1+n_2\lambda_2+n_3\lambda_3+n_4\lambda_4$ be a dominant integral weight.

(a) If $n_3=n_4=0$, then $N^\mu_\mu=0$.

(b) If $n_3>0$ and $n_4=0$, then $N^\mu_\mu=1$ and $K_\gk^\mu(\lambda_4)[0]=\mbb C\cdot F_{\rho_4}v_{\rho_4}$.

(c) If $n_3=0$ and $n_4>0$, then $N^\mu_\mu=1$ and $K_\gk^\mu(\lambda_4)[0]=\mbb C\cdot F_{\rho_3}v_{\rho_3}$.

(d) If $n_3,n_4>0$, then $N^\mu_\mu=2$.
\end{lm}

We shall prove this lemma together with the following one:

\begin{lm}\label{lb1}
	$F_{\rho_3}v_{\rho_3}$ and $F_{\rho_4}v_{\rho_4}$ form a basis of $L_\gk(\lambda_4)[0]$.
\end{lm}

\begin{proof}
We know that any vector in $K_\gk^\mu(\lambda_4)[0]$ must be of the form $F_\alpha^{n_{\mu,\alpha}+1}u$ where $u\in L_\gk(\lambda_4)$ has weight $(n_{\mu,\alpha}+1)\alpha$ and $\alpha$ is one of $\rho_1,\dots,\rho_4$. Since $(n_{\mu,\alpha}+1)\alpha$ is one of the 25  weights of $L_\gk(\lambda_4)$, the only possible case is that $n_{\mu,\alpha}=0$ (equivalently, $(\mu|\alpha)=0$) and $\alpha\in\{\rho_3,\rho_4\}$. (Note that $\rho_1,\rho_2$ are not weights of $L_\gk(\lambda_4)$.)

It is easy to calculate that
\begin{gather*}
(\mu|\rho_3)=\frac {n_3}2 ,\qquad (\mu|\rho_4)=\frac {n_4}2.
\end{gather*}
In case (a), $\mu$ is orthogonal to both $\rho_3,\rho_4$. It follows that $K_\gk^\mu(\lambda_4)[0]$ is spanned by $F_{\rho_3}v_{\rho_3},F_{\rho_4}v_{\rho_4}$. In particular, $K^0_\gk(\lambda_4)[0]$  is spanned by the two vectors. Since $\Hom_\gk(\lambda_4\otimes 0,0)$ is clearly trivial, by \eqref{eq2}, we must have $K_\gk^0(\lambda_4)[0]=L_\gk(\lambda)[0]$ whose dimension is $2$. Thus, $F_{\rho_3}v_{\rho_3},F_{\rho_4}v_{\rho_4}$ span $L_\gk(\lambda_4)[0]$.  This proves lemma \ref{lb1}. By this lemma, $K_\gk^\mu(\lambda_4)[0]$, which has dimension $2$,  equals $L_\gk(\lambda_4)[0]$. So $N^\mu_\mu=0$.

Assume case (b). Then $\mu$ is orthogonal to $\rho_4$ but not $\rho_3$. It follows that $K_\gk^\mu(\lambda_4)[0]=\mbb C\cdot F_{\rho_4}v_{\rho_4}$. Similarly, in case (c), $\mu$ is orthogonal to $\rho_3$ but not $\rho_4$. So $K_\gk^\mu(\lambda_4)[0]=\mbb C\cdot F_{\rho_3}v_{\rho_3}$. Finally, in case (d), neither $\rho_3$ nor $\rho_4$ is orthogonal to $\mu$. So $K^\mu_\gk(\lambda_4)[0]$ must be trivial.
\end{proof}

Similar to lemma \ref{lb1}, we have:

\begin{lm}\label{lb15}
Let $\rho_5=\frac 12[-1,1,1,-1]$ whose negative is a root in group A. Choose a non-zero $v_{\rho_5}\in L_\gk(\lambda_4)[\rho_5]$. Then $F_{\rho_3}v_{\rho_3}$ and $F_{\rho_5}v_{\rho_5}$ form a basis of $L_\gk(\lambda_4)[0]$.
\end{lm}

\begin{proof}
For each root $\alpha$, let $\varpi_\alpha$ be the element in the Weyl group defined by the reflection along the hyperplane orthogonal to $\alpha$. In other words, for each $\lambda\in\fk h$, $\varpi_\alpha(\lambda)=\lambda-n_{\lambda,\alpha}\cdot \alpha$. Now we set $\alpha=[0,1,1,0]$ and $\beta=[1,0,0,0]$. Set $\varpi=\varpi_\alpha\varpi_\beta$. Then $\varpi(\rho_3)=\rho_3$ and $\varpi(\rho_4)=\rho_5$. Thus $F_{\rho_3}v_{\rho_3}$ and $F_{\rho_5}v_{\rho_5}$ are linearly independent since $F_{\rho_3}v_{\rho_3}$ and $F_{\rho_4}v_{\rho_4}$ are so. Indeed, one can show that two vectors are not parallel by calculating the angle between them. Due to the equivalence induced by reflections, the angles between $F_{\rho_3}v_{\rho_3},F_{\rho_5}v_{\rho_5}$ and between $F_{\rho_3}v_{\rho_3},F_{\rho_4}v_{\rho_4}$ can be calculated using the same algorithm and share similar properties. (See section \ref{lb10} for an instance of calculating the angle.) 
\end{proof}

When $\nu-\mu$ is a non-zero weight of $L_\gk(\lambda_4)$, the weight space $L_\gk(\lambda_4)[\nu-\mu]$ has dimension $1$. In this case, the number $N^\nu_\mu$ can be calculated by the following method. (Cf. \cite{Gui19b} corollary 1.18)

\begin{pp}\label{lb3}
Let $\mu,\nu\in P_+(\gk)$. Assume that $\nu-\mu$  is a non-zero weight of $L_\gk(\lambda_4)$. Then $N^\nu_\mu=1$ if and only if for any $\alpha\in\{\rho_1,\rho_2,\rho_3,\rho_4\}$,
\begin{align*}
\dim L_\gk(\lambda_4)[\nu-\mu+(n_{\mu,\alpha}+1)\alpha]=0.
\end{align*}
Otherwise, $N^\nu_\mu=0$.
\end{pp}

\section{The fundamental types}\label{lb11}

We say that an intertwining operator in $\mc V^l_\gk{\nu\choose\lambda_4~\mu}$ is of type $\nu\choose\mu$ level $l$.

We have shown that the level $1$ intertwining operators with charge space $L_\gk(\lambda_4,1)$ are energy-bounded. Thus, by  \cite{Gui19b} proposition 2.14 and that $(\lambda_4|\theta)=1$, to prove the energy bounds condition for any type $\nu\choose\mu$ level $l$ intertwining operator, it suffices to consider the case that $(\mu|\theta)=(\nu|\theta)=l$. In this case, $\nu-\mu$, which is orthogonal to $\theta$, is either $0$ or one of the six roots in group A. Also, since $N_{\lambda_4\lambda_4}^0=1$ (by, for example, proposition \ref{lb3}), $L_\gk(\lambda_4,l)$ is self-dual. Thus any type ${\nu\choose\mu}$ level $l$ intertwining operator is the adjoint of a type ${\mu\choose\nu}$ level $l$ intertwining operator.  (See \cite{Gui19a} section 1.3 for the definition of adjoint intertwining operators.) Thus, by \cite{Gui19a} corollary 3.7-(d), type ${\nu\choose\mu}$ level $l$ intertwining operators are energy-bounded if and only if type $\mu\choose\nu$ level $l$ intertwining operators are so.

\begin{df}\label{lb21}
Assume $k,l\in\mbb Z_+$ and $k\leq l$. Let $\mu_0,\nu_0\in P_+(\gk,k)$ and $\mu,\nu\in P_+(\gk,l)$. We say that type $\nu\choose\mu$ level $l$ \textbf{reduces to} type $\nu_0\choose\mu_0$ level $k$, if the following conditions are satisfied:

(a) $N^\nu_\mu\leq N^{\nu_0}_{\mu_0}$.

(b) There exists $\rho\in P_+(\gk,l-k)$ such that $\mu=\mu_0+\rho$ and $\nu=\nu_0+\rho$.
\end{df}

The following result generalizes \cite{Gui19b} lemma 2.15-(a), and applies to all Lie types but not just $\fk f_4$.

\begin{pp}\label{lb8}
Suppose that type $\nu\choose\mu$ level $l$ reduces to type $\nu_0\choose\mu_0$ level $k$. Then $N^\nu_\mu=N^{\nu_0}_{\mu_0}$. Moreover, all type $\nu\choose\mu$ level $l$ intertwining operators are energy-bounded if all type $\nu_0\choose\mu_0$ level $k$ intertwining operators are so.
\end{pp}

\begin{proof}
The main idea is the same as in the proof of \cite{Gui19b} lemma 2.15-(a). Define a linear map $\mc V^k_\gk{\nu_0\choose\lambda_4~\mu_0}\rightarrow \mc V^l_\gk{\nu\choose\lambda_4~\mu}$, $\mc Y\mapsto \wtd{\mc Y}$ as follows. Set $k'=l-k$. Set irreducible unitary $V^k_\gk\otimes V^{k'}_\gk$-modules
\begin{gather*}
W_1=L_\gk(\lambda_4,k)\otimes L_\gk(0,k'),\\ W_2=L_\gk(\mu_0,k)\otimes L_\gk(\rho,k'),\\ W_3=L_\gk(\nu_0,k)\otimes L_\gk(\rho,k').
\end{gather*}
Note that the vertex operator $Y_\rho$ associated to the $V^{k'}_\gk$-module $L_\gk(\rho,k')$ is of type $\rho\choose 0~\rho$. For each $\mc Y$ of type $\nu_0\choose\mu_0$ level $k$, 
\begin{align*}
\fk Y:=\mc Y\otimes Y_\rho
\end{align*}
is an intertwining operator of $V^k_\gk\otimes V^{k'}_\gk$ of type $W_3\choose W_1 W_2$. Since $V_\gk^{k'}$ is strongly energy-bounded, the vertex operator $Y_\rho$ is energy-bounded. So $\fk Y$ is energy-bounded if $\mc Y$ is so.

Let $v_{\lambda_4},v_{\mu_0},v_{\nu_0},\Omega,v_\rho$ be the highest weight (and lowest energy) vectors of $L_\gk(\lambda_4,k),L_\gk(\mu_0,k),L_\gk(\nu_0,k),L_\gk(0,k'),L_\gk(\rho,k')$. (Note that $\Omega$ can be chosen to be the vacuum vector of the VOA $V^{k'}_\gk=L_\gk(0,k')$.) Set 
\begin{align*}
w_1=v_{\lambda_4}\otimes \Omega,\qquad w_2=v_{\mu_0}\otimes v_\rho,\qquad w_3=v_{\nu_0}\otimes v_\rho
\end{align*}
which are homogeneous vectors  in $W_1,W_2,W_3$ respectively. Consider the diagonal embedding $\gk\subset \gk\oplus\gk$ which induces $V^l_\gk\subset V^k_\gk\otimes V^{k'}_\gk$. Then $w_1,w_2,w_3$ are highest weight vectors of the level $l$ affine Lie algebra $\wht\gk_l$ with weights $\lambda_4,\mu,\nu$ respectively. Let $\wtd{\mc Y}$ be the ``restriction" of $\fk Y$ to the irreducible $V^l_\gk$-submodules generated by $w_1,w_2,w_3$ respectively.  To be more precise, note that the unitary  $\gk$-modules $\gk w_1,\gk w_2,\gk w_3$ are equivalent to $L_\gk(\lambda_4),L_\gk(\mu),L_\gk(\nu)$ respectively. Let $e_3$ be the orthogonal projection of $W_3$ onto $\gk w_3$. Let $s=\Delta_{\lambda_4\mu_0}^{\nu_0}-1$, and define $T_{\mc Y}\in\Hom_\gk(\lambda_4\otimes\mu,\nu)$ by
\begin{gather*}
T_{\mc Y}:L_\gk(\lambda_4)\otimes L_\gk(\mu)\rightarrow L_\gk(\nu)\\
u^{(\lambda_4)}\otimes u^{(\mu)}\mapsto e_3\fk Y(u^{(\lambda_4)})_s u^{(\mu)}.
\end{gather*} Then, by \cite{Gui19b} theorem 2.12, there is a  unique $\wtd{\mc Y}$ of type $\nu\choose\mu$ level $l$ such that $\Psi\wtd{\mc Y}=T_{\mc Y}$, (Recall that the injective map $\Psi$ is defined in \eqref{eq4}.) and $\wtd{\mc Y}$ is energy-bounded if $\mc Y$ is so. (Indeed, by the proof of that theorem, $\wtd {\mc Y}$ is chosen to be $z^a\cdot p_3\fk Y(p_1,z)p_2$ for some $a\in\mbb R$, where $p_1,p_2,p_3$ are the orthogonal projections of $W_3$ onto $\wht\gk w_1,\wht\gk w_2,\wht\gk w_3$ respectively.)

From the definition of $\fk Y$ and $T_{\mc Y}$, it is not hard to see that $\Gamma\Psi\mc Y$ equals $\Gamma T_{\mc Y}$. Therefore, as $\Psi\wtd{\mc Y}=T_{\mc Y}$, we have $\Gamma\Psi\mc Y=\Gamma\Psi\wtd{\mc Y}$. Thus, by the injectivity of $\Psi$ and $\Gamma$, we have $\mc Y=0$ if and only if $\wtd{\mc Y}=0$. Therefore the linear map $\mc Y\mapsto\wtd{\mc Y}$ is injective. By condition (a), this map is bijective.
\end{proof}

\begin{pp}\label{lb7}
Assume that $(\nu|\theta)=(\mu|\theta)=l$. Then any intertwining operator of type $\nu\choose\mu$ level $l$ reduces to one of the following \textbf{fundamental types}:
\begin{enumerate}
	\renewcommand*\labelenumi{(\theenumi)}
	\item Type $2\lambda_3\choose\lambda_2+\lambda_4$ level $4$ and its adjoint.
	\item Type $\lambda_2\choose \lambda_1+\lambda_4$ level $3$ and its adjoint.
	\item Type $2\lambda_4\choose \lambda_3$ level $2$ and its adjoint.
	\item Type $\lambda_3+\lambda_4\choose\lambda_2$ level $3$ and its adjoint.
	\item Type $\lambda_2+\lambda_4\choose\lambda_1+\lambda_3$ level $4$ and its adjoint.
	\item Type $\lambda_3\choose\lambda_1$ level $2$ and its adjoint.
	\item Type $\lambda_3+\lambda_4\choose \lambda_3+\lambda_4$ level $3$.
	\item Type $\lambda_3\choose\lambda_3$ level $2$.
	\item Type $\lambda_4\choose\lambda_4$ level $1$.
\end{enumerate}
\end{pp}

\begin{proof}
Assume $N^\nu_\mu>0$. By lemma \ref{lb5}, if $\nu=\mu$, then any type reduces to one of (7), (8), (9). Now assume $\nu\neq \mu$. Then $N^\nu_\mu\leq 1$, and  either $\nu-\mu$ or its negative is one of the six roots in group A. (See section \ref{lb6}.) By proposition \ref{lb3} or by LieART computations (see section \ref{lb14}), it is not hard to check that each of the first 6 cases has fusion/tensor product rule $1$. Moreover, these 6 cases correspond to the six positive roots in group A. Assume $\nu-\mu$ is positive. So $\nu-\mu$ is a positive root in group A. Then $\nu\choose\mu$ level $l$ reduces the one of (1)-(6) corresponding to $\nu-\mu$. If $\nu-\mu$ is negative, then $\mu\choose\nu$, which is adjoint to $\nu\choose\mu$, reduces to the one of (1)-(6).
\end{proof}

Note that only  case (7) has fusion rule $2$. All the other fundamental types have fusion rule $1$ by (for instance) proposition \ref{lb3}.

We close this section with an application of proposition \ref{lb8}. Given $\lambda,\mu\in P_+(\fk f_4,l)$, one can define the fusion product
\begin{align*}
L_{\fk f_4}(\lambda,l)\boxtimes L_{\fk f_4}(\mu,l)=\bigoplus_{\nu\in P_+(\fk f_4,l)} L_{\fk f_4}(\lambda,l)\otimes \mc V_{\fk f_4}^l{\nu\choose\lambda~\mu}^*,
\end{align*}
which we write as $\lambda\boxtimes\mu$ for short. (Note that $\lambda\boxtimes\mu$ depends on the level $l$.)  Thus, the multiplicity of $\nu$ in $\lambda\boxtimes\mu$ is the fusion rule $N^\nu_{\lambda\mu}$. Fusion products of more than two irreducible modules can be defined inductively. The associativity of $\boxtimes$ (i.e. $(\lambda\boxtimes\mu)\boxtimes\nu\simeq\lambda\boxtimes(\mu\boxtimes\nu)$) follows from, say, \cite{TUY89} or \cite{Hua95,Hua05} or \cite{NT05}.

\begin{thm}\label{lb22}
	For any $l=1,2,3,\dots$, the irreducible $V_{\fk f_4}^l$-module $L_{\fk f_4}(\lambda_4,l)$ $\boxtimes$-generates the category of (semisimple) $V_{\fk f_4}^l$-modules. Namely, any irreducible $V_{\fk f_4}^l$-module is equivalent to a submodule of a (finite) fusion product of $L_{\fk f_4}(\lambda_4,l)$.
\end{thm}

\begin{proof}
	Let $Q$ be a subset of $P_+(\fk f_4,l)$ consisting of all $\nu$ such that $L_{\fk f_4}(\nu,l)$ appears as an irreducible submodule of a  fusion product of $L_{\fk f_4}(\lambda_4,l)$. Equivalently, $Q$ is the smallest subset containing $\lambda_4$ and satisfying that if $\mu\in Q$ and $N^\nu_\mu\equiv N^\nu_{\lambda_4~\mu}>0$ then $\nu\in Q$. Note that $0\in Q$ since $\lambda_4$ is self-dual. 
	
We first claim that if $\lambda,\mu\in Q$ and $\lambda+\mu$ is admissible at level $l$ then $\lambda+\mu\in Q$. By the associativity of $\boxtimes$, it suffices to show that $N_{\lambda~\mu}^{\lambda+\mu}>0$. Suppose $N_{\lambda~\mu}^{\lambda+\mu}=0$. Set $k=(\theta|\lambda),\nu=\lambda+\mu,\mu_0=0,\nu_0=\lambda$. Then, type ${\nu\choose\lambda~\mu}$ level $l$ reduces to type $\lambda\choose \lambda~0$ level $k$ in the same sense of definition \ref{lb21}, namely, we have (a) $N_{\lambda\mu}^\nu\leq N^{\nu_0}_{\lambda \mu_0}$, and (b) there exists $\rho\in P_+(\fk f_4,l-k)$ such that $\mu=\mu_0+\rho$ and $\nu=\nu_0+\rho$. (Set $\rho=\mu$.) Thus, by the proof of proposition \ref{lb8}, we have $N^\nu_{\lambda\mu}=N^{\nu_0}_{\lambda\mu_0}$, i.e., $N^{\lambda+\mu}_{\lambda~\mu}=N^\lambda_{\lambda~0}$, which equals $1$. This gives a contradiction.

We now use the above result to prove $Q=P_+(\fk f_4,l)$. This is clearly true when $l=1$ (since both sets equal $\{0,\lambda_4\}$). Assume $l\geq 2$. Then, by the above result, $2\lambda_4\in Q$. By for instance proposition \ref{lb3}, we have $N^{\lambda_3}_{2\lambda_4}\equiv N^{\lambda_3}_{\lambda_4~2\lambda_4}=1$ (cf. Prop. \ref{lb7}-(3)). So $\lambda_3\in Q$. Since $N^{\lambda_1}_{\lambda_3}=1$ (cf. Prop. \ref{lb7}-(6)), $\lambda_1\in Q$. Thus $Q=P_+(\fk f_4,l)=\{0,\lambda_4,2\lambda_4,\lambda_1,\lambda_3\}$  when $l=2$. Assume $l\geq 3$. Then $\lambda_1+\lambda_4\in Q$ by the above paragraph. Thus, as $N^{\lambda_2}_{\lambda_1+\lambda_4}=1$ (cf. Prop. \ref{lb7}-(2)), we conclude $\lambda_2\in Q$. By the above paragraph, any $n_1\lambda_1+\cdots+n_4\lambda_4$, if admissible at level $l$, is in $Q$. This proves $Q=P_+(\fk f_4,l)$ in general.
\end{proof}

\section{Proof for the fundamental types}\label{lb19}

By the results in the last section, to prove theorem \ref{lb2}, it suffices to prove the energy bounds condition for any intertwining operator of fundamental type. Since the case $l\leq 2$ has already been proved, it remains to prove the cases (1) (2) (4) (5) (7) in proposition \ref{lb7}. We first discuss the easier case.

\subsubsection*{Proof for case (7)}

\begin{pp}\label{lb18}
The intertwining operators of type $\lambda_3+\lambda_4\choose\lambda_3+\lambda_4$ level $3$ are energy-bounded.
\end{pp}

\begin{proof}
Choose a non-zero intertwining operator $\mc Y_1$ of type $\lambda_3\choose\lambda_3$ level $2$. Then $\mc Y_1$ is energy-bounded by proposition \ref{lb4}. As in the proof of proposition \ref{lb8},  one can construct $\wtd{\mc Y}_1$ of type $\lambda_3+\lambda_4\choose\lambda_3+\lambda_4$ level $3$ by compressing $\mc Y_1\otimes Y_{\lambda_4}$, where $Y_{\lambda_4}$ is the vertex operator of $L_\gk(\lambda_4,1)$. Moreover, $\wtd{\mc Y}_1$ is energy-bounded, and by the last paragraph of that proof, $\Gamma\Psi \wtd{\mc Y}_1=\Gamma\Psi\mc Y_1$. Now, $\Psi\mc Y_1$ is a non-zero element in $\Hom_\gk(\lambda_4\otimes\lambda_3,\lambda_3)$. Thus, by \eqref{eq1}, $\Gamma\Psi \wtd{\mc Y}_1$ is a non-zero linear functional on $L_\gk(\lambda_4)[0]$ orthogonal to $K_\gk^{\mu_3}(\lambda_3)[0]$. Therefore, by lemma \ref{lb5},  $\Gamma\Psi \wtd{\mc Y}_1$ kills $F_{\rho_4}v_{\rho_4}$.

Similarly, we choose a non-zero $\mc Y_2$ of type $\lambda_4\choose\lambda_4$ level $1$. Then one can construct $\wtd{\mc Y}_2$ of type $\lambda_3+\lambda_4\choose\lambda_3+\lambda_4$ level $3$ by compressing $\mc Y_2\otimes Y_{\lambda_3}$ where $Y_{\lambda_3}$ is the vertex operator of $L_\gk(\lambda_3,2)$. Then $\wtd{\mc Y}_2$ is energy-bounded, and, by \eqref{eq1} and lemma \ref{lb5}, $\Gamma\Psi \wtd{\mc Y}_2$ kills $F_{\rho_3}v_{\rho_3}$. Therefore, by lemma \ref{lb1}, $\wtd{\mc Y}_1$ and $\wtd{\mc Y}_2$ are linearly independent. So the $2$-dimensional vector space $\mc V_\gk^3{\lambda_3+\lambda_4\choose\lambda_4~\lambda_3+\lambda_4}$ is spanned by the energy-bounded intertwining operators $\wtd{\mc Y}_1$  and $\wtd{\mc Y}_2$. This completes the proof.
\end{proof}

\subsubsection*{Proof for cases (1) (2) (4) (5)}

Our proof of these four cases relies on the following proposition which holds for any complex (unitary) finite dimensional simple Lie algebra. The statement and the proof are similar to (but slightly more general than) Lemma 2.15-(b) of \cite{Gui19b}. It is also a generalization of proposition \ref{lb8}. 

\begin{pp}\label{lb9}
Choose $\lambda,\mu,\nu\in P_+(\gk,l)$ satisfying $\dim\mc V_\gk^l{\nu\choose\lambda~\mu}>0$. Assume that there exist $\rho,\mu_0,\nu_0\in P_+(\gk,l)$ satisfying the following conditions. (We set $k=\max\{(\lambda|\theta),(\mu_0|\theta),(\nu_0|\theta)\}$.)

(a) $\dim\mc V_\gk^l{\nu\choose\lambda~\mu}\leq \dim\mc V_\gk^k{\nu_0\choose\lambda~\mu_0}$. 

(b) $\mu=\mu_0+\rho$, and $\dim L_\gk(\nu_0)[\nu-\rho]=\dim\Hom_\gk(\nu_0\otimes\rho,\nu)=1$.

(c) $(\rho|\theta)+k\leq l$.

(d) For any non-zero $\mc Y\in \mc V_\gk^k{\nu_0\choose\lambda~\mu_0}$,
\begin{align}
(\Psi\mc Y)\Big(L_\gk(\lambda)[\nu-\mu]\otimes L_\gk(\mu_0)[\mu_0]\Big)\neq 0.\label{eq6}
\end{align}
(Note that the above expression is a subspace of $L_\gk(\nu_0)[\nu-\rho]$.) Then $\dim\mc V_\gk^l{\nu\choose\lambda~\mu}=\dim\mc V_\gk^k{\nu_0\choose\lambda~\mu_0}$, and all the intertwining operators in $\mc V_\gk^l{\nu\choose\lambda~\mu}$ are energy-bounded if those in $\mc V_\gk^k{\nu_0\choose\lambda~\mu_0}$ are so.
\end{pp}

\begin{proof}
As in the proof of proposition \ref{lb8}, define a linear map $\mc V^k_\gk{\nu_0\choose\lambda~\mu_0}\rightarrow \mc V^l_\gk{\nu\choose\lambda~\mu}$, $\mc Y\mapsto \wtd{\mc Y}$ as follows. Set $k'=l-k$. Then, by (b), $\rho\in P_+(\gk,k')$. Set irreducible unitary $V^k_\gk\otimes V^{k'}_\gk$-modules
\begin{gather*}
W_1=L_\gk(\lambda,k)\otimes L_\gk(0,k'),\\ W_2=L_\gk(\mu_0,k)\otimes L_\gk(\rho,k'),\\ W_3=L_\gk(\nu_0,k)\otimes L_\gk(\rho,k').
\end{gather*}
For each $\mc Y\in \mc V^k_\gk{\nu_0\choose\lambda~\mu_0}$, 
\begin{align*}
\fk Y:=\mc Y\otimes Y_\rho
\end{align*}
is an intertwining operator of $V^k_\gk\otimes V^{k'}_\gk$ of type $W_3\choose W_1 W_2$.

We now define the homogeneous vectors $w_1,w_2,w_3$ as in the proof of proposition \ref{lb8}. Let $v_\lambda,v_{\mu_0},v_\rho$ be the highest weight (and lowest energy) vectors of $L_\gk(\lambda,k),L_\gk(\mu_0,k),L_\gk(\rho,k')$ respectively.  $\Omega$ is the vacuum vector and highest weight vector of $L_\gk(0,k')$. Set
\begin{gather*}
w_1=v_\lambda\otimes\Omega,\qquad w_2=v_{\mu_0}\otimes v_\rho.
\end{gather*}
The lowest ($L_0$-) energy subspace of $W_3$ is $L_\gk(\nu_0)\otimes L_\gk(\rho)$. By (b), the $\gk$-module $L_\gk(\nu_0)\otimes L_\gk(\rho)$ has a unique irreducible submodule (equivalent to) $L_\gk(\nu)$. We let $w_3$ be a (non-zero) highest weight vector of $L_\gk(\nu)$. Consider again the diagonal embedding $V_\gk^l\subset V_\gk^k\otimes V_\gk^{k'}$. Then $w_1,w_2,w_3$ are highest weight vectors of the affine Lie algebra $\wht\gk_l$ with weights $\lambda,\mu,\nu$ respectively. Let $e_3$ be the orthogonal projection of $W_3$ onto $\gk w_3$.  Let $s=\Delta_{\lambda\mu_0}^{\nu_0}-1$, and define $T_{\mc Y}\in\Hom_\gk(\lambda\otimes\mu,\nu)$ by
\begin{gather*}
T_{\mc Y}:L_\gk(\lambda)\otimes L_\gk(\mu)\rightarrow L_\gk(\nu)\\
u^{(\lambda)}\otimes u^{(\mu)}\mapsto e_3\fk Y(u^{(\lambda)})_s u^{(\mu)}.
\end{gather*} 
Again, by \cite{Gui19b} theorem 2.12, there is a  unique $\wtd{\mc Y}\in\mc V_\gk^l{\nu\choose\lambda~\mu}$ such that $\Psi\wtd{\mc Y}=T_{\mc Y}$,  and that $\wtd{\mc Y}$ is energy-bounded if $\mc Y$ is so.

It remains to check that $\mc Y\mapsto\wtd{\mc Y}$ is injective. Suppose $\wtd{\mc Y}=0$. Then $T_{\mc Y}=0$. Identify $L_\gk(\lambda,k)$ with $L_\gk(\lambda,k)\otimes\Omega$ and hence $L_\gk(\lambda)$ with $L_\gk(\lambda)\otimes\Omega$.  For each $u\in L_\gk(\lambda)[\nu-\mu]$, 
\begin{align*}
T_{\mc Y}(u\otimes w_2)=e_3(\mc Y\otimes Y_\rho)(u\otimes\Omega)_sw_2=e_3(\mc Y(u)_s v_{\mu_0}\otimes v_\rho),
\end{align*}
which is $0$. Since $\mc Y(u)_s v_{\mu_0}\otimes v_\rho\in L_\gk(\nu_0)\otimes v_\rho$, $e_3$ restricts to and can be regarded as the projection of $L_\gk(\nu_0)\otimes L_\gk(\rho)$ onto $L_\gk(\nu)\simeq \gk w_3$. Thus $e_3$ is a non-zero element in $\Hom_\gk(\nu_0\otimes\rho,\nu)$. Since $\mc Y(u)_s v_{\mu_0}\in L_\gk(\nu_0)[\nu-\rho]$, one has $e_3(\mc Y(u)_s v_{\mu_0}\otimes v_\rho)=(\Gamma e_3)(\mc Y(u)_s v_{\mu_0})$, which equals $0$. Since $\Gamma$ is injective, $\Gamma e_3\neq 0$. So $\mc Y(u)_s v_{\mu_0}=0$ since $L_\gk(\nu_0)[\nu-\rho]$ is one dimensional. To summarize, we have proved that $\mc Y(u)_s v_{\mu_0}=0$ for any $u\in L_\gk(\lambda)[\nu-\mu]$. Therefore, by condition (d), we must have $\mc Y=0$.
\end{proof}

\begin{lm}\label{lb12}
In proposition \ref{lb9}, we set $\alpha=\nu_0-\nu+\rho$ and $\eta=\nu-\mu$. Then condition (d) holds if the following are satisfied:

(i) $\dim \Hom_\gk(\lambda\otimes\mu_0,\nu_0)=1$.

(ii) $\alpha$ is a positive root of $\gk$.

(iii) There exists $u\in L_\gk(\lambda)[\eta]$ such that $E_\alpha u\notin K_\gk^{\mu_0}(\lambda)[\nu_0-\mu_0]$. \\
Moreover, assume (i) and (ii). Then (iii) holds if the following is satisfied:

(iii') $\dim L_\gk(\lambda)[\nu_0-\mu_0]=1$ and $(\eta|\alpha)<0$.

\end{lm}

We will show that cases (1) and (4) satisfy (i) (ii) (iii), and that cases (2) and (5) satisfy (i) (ii) (iii').

\begin{proof}
Choose any non-zero $\mc Y\in\mc V_\gk^k{\nu_0\choose\lambda~\mu_0}$. Then $T:=\Psi\mc Y$ is non-zero. Choose any non-zero $u\in L_\gk(\lambda)[\eta]$ satisfying condition (iii). As usual, we let $v_{\mu_0}$ and $v_{\nu_0}$ be (non-zero) highest weight vectors of $L_\gk(\mu_0),L_\gk(\nu_0)$ respectively. Since $\alpha$ is a positive root, we have a lowering operator $F_\alpha$. Then
\begin{align*}
\bk{T(u\otimes v_{\mu_0})|F_\alpha v_{\nu_0}}=\bk{T(E_\alpha u\otimes v_{\mu_0})|v_{\nu_0}}+\bk{T(u\otimes E_\alpha v_{\mu_0})|v_{\nu_0}}=\bk{T(E_\alpha u\otimes v_{\mu_0})|v_{\nu_0}}.
\end{align*}
Since $E_\alpha u$ has weight $\eta+\alpha=\nu_0-(\mu-\rho)=\nu_0-\mu_0$, we have
\begin{align*}
\bk{T(u\otimes v_{\mu_0})|F_\alpha v_{\nu_0}}=(\Gamma T)(E_\alpha u).
\end{align*}
By (i) and \eqref{eq1}, $K_\gk^{\mu_0}(\lambda)[\nu_0-\mu_0]$ has codimension $1$, which must be the kernal of the nonzero linear functional $\Gamma T$. So $(\Gamma T)(E_\alpha u)$ is not zero since $E_\alpha u$ is not in this kernal. This proves \eqref{eq6}.

Assume that (i), (ii), and (iii') are true. Since we have assumed $\dim\mc V_\gk^l{\nu\choose\lambda~\mu}>0$ in proposition \ref{lb9}, $L_\gk(\lambda)[\nu-\mu]=L_\gk(\lambda)[\eta]$ is non-trivial. Choose any non-zero $u\in L_\gk(\lambda)[\eta]$.  Since $(\eta|\alpha)<0$ and hence $n_{\eta,\alpha}<0$, we have $E_\alpha u\neq 0$. Indeed, this follows either by a standard $\fk{sl}_2$-argument or by the following calculation:
\begin{align*}
&\bk{E_\alpha u|E_\alpha u}=\bk{F_\alpha E_\alpha u|u}=\bk{F_\alpha u|F_\alpha u}-\bk{H_\alpha u|u}\\
\geq &-\bk{H_\alpha u|u}=-n_{\eta,\alpha}(u|u)>0.
\end{align*}
Since $\dim L_\gk(\lambda)[\nu_0-\mu_0]=\dim\Hom_\gk(\lambda\otimes\mu_0,\nu_0)=1$, the subspace $K_\gk^{\mu_0}(\lambda)[\nu_0-\mu_0]$ is trivial by \eqref{eq2}. This proves (iii).
\end{proof}

We return to the Lie algebra $\gk=\fk f_4$. We now prove the energy bounds condition for the cases (1) (2) (4) (5). To be more specific, we prove:

\begin{pp}\label{lb13}
The intertwining operators of the following types are energy-bounded:\\[1ex]
$~~~~~~~$(1) Type $\lambda_2+\lambda_4\choose 2\lambda_3$ level $4$.\\[1ex]
$~~~~~~~$(2) Type $\lambda_2\choose \lambda_1+\lambda_4$ level $3$.\\[1ex]
$~~~~~~~$(4) Type $\lambda_2\choose\lambda_3+\lambda_4$ level $3$.\\[1ex]
$~~~~~~~$(5) Type $\lambda_2+\lambda_4\choose\lambda_1+\lambda_3$ level $4$.
\end{pp}
Then their adjoint types are also energy-bounded. (See the beginning of section \ref{lb11}.) 

\begin{proof}
For all these cases, we let $\lambda=\lambda_4$. Recall the notations in section \ref{lb6}. Choose $\mu,\nu,\rho,\mu_0,\nu_0$, and calculate $\nu-\rho$,
\begin{gather*}
k=\max\{(\lambda|\theta),(\mu_0|\theta),(\nu_0|\theta)\},\qquad\alpha=\nu_0-\nu+\rho,\qquad\eta=\nu-\mu.
\end{gather*}
\textbf{Case (1)}, $k=2$.
\begin{gather*}
\mu=2\lambda_3,\qquad\nu=\lambda_2+\lambda_4,\qquad \rho=\lambda_3,\qquad \mu_0=\lambda_3,\qquad \nu_0=\lambda_3,\\
\nu-\rho=(0,1,-1,1),\qquad  \alpha=[0,0,0,1],\qquad \eta=[0,0,0,-1].
\end{gather*}
\textbf{Case (2)}, $k=2$.
\begin{gather*}
\mu=\lambda_1+\lambda_4,\qquad\nu=\lambda_2,\qquad \rho=\lambda_4,\qquad\mu_0=\lambda_1,\qquad \nu_0=\lambda_3,\\
\nu-\rho=(0,1,0,-1),\qquad  \alpha=\frac 12[1,-1,-1,1],\qquad \eta=[0,0,1,0].
\end{gather*}
\textbf{Case (4)}, $k=1$.
\begin{gather*}
\mu=\lambda_3+\lambda_4,\qquad\nu=\lambda_2,\qquad \rho=\lambda_3,\qquad\mu_0=\lambda_4,\qquad \nu_0=\lambda_4,\\
\nu-\rho=\frac 12[1,1,1,-1],\qquad  \alpha=\frac 12[1,-1,-1,1],\qquad \eta=\frac 12[-1,1,1,-1].
\end{gather*}
\textbf{Case (5)}, $k=2$.
\begin{gather*}
\mu=\lambda_1+\lambda_3,\qquad\nu=\lambda_2+\lambda_4,\qquad\rho=\lambda_3,\qquad \mu_0=\lambda_1,\qquad \nu_0=\lambda_3,\\
\nu-\rho=(0,1,-1,1),\qquad  \alpha=[0,0,0,1],\qquad \eta=\frac 12[1,-1,1,-1].
\end{gather*}

We know that $N^\nu_\mu=1$ since all the cases in proposition \ref{lb7}, except case (7), have fusion (tensor product) rule $1$. Thus $\dim\mc V_\gk^l{\nu\choose\lambda~\mu}=1$. We now check that conditions (a) (b) (c) of proposition \ref{lb9} and (i) (ii) (iii) (or (iii')) of lemma \ref{lb11} are satisfied. By the fusion rules in proposition \ref{lb7}, in each case, $N^{\nu_0}_{\mu_0}=\dim\mc V_\gk^k{\nu_0\choose\lambda~\mu_0}=1$. Thus (a) and (i) are satisfied. (c) and (ii) are obvious. In case (4), $\nu-\rho$ is in group B, which is a weight $L_\gk(\nu_0)=L_\gk(\lambda_4)$ with multiplicity $1$. So $\dim L_\gk(\nu_0)[\nu-\rho]=1$. One can show  that $\dim\Hom_\gk(\nu_0\otimes\rho,\nu)=1$ using proposition \ref{lb3}. This proves (b). For the other cases, condition (b) can be checked using LieART. (See section \ref{lb14} for details.)

It remains to check condition (iii) or (iii') of lemma \ref{lb12}. We first discuss cases (2) and (5). Then $\nu_0-\mu_0=(-1,0,1,0)$ is a non-zero weight of $L_\gk(\lambda)=L_\gk(\lambda_4)$. So $\dim L_\gk(\lambda)[\nu_0-\mu_0]=1$. It is easy to check that $(\eta|\alpha)<0$. This proves (iii').  Now, we assume case (1). Recall the definition of $v_{\rho_3}$ and $v_{\rho_4}$ in \eqref{eq7}. Then $K_\gk^{\mu_0}(\lambda)[\nu_0-\mu_0]=K_\gk^{\lambda_3}(\lambda_4)[0]$, which, by lemma \ref{lb5}, is spanned by $F_{\rho_4}v_{\rho_4}$. We have $\eta=-\alpha=-\rho_3$. Choose any non-zero $u\in L_\gk(\lambda)[\eta]=L_\gk(\lambda_4)[-\rho_3]$. Then $E_\alpha u=E_{\rho_3}u$, which has weight $0$. By a standard $\fk{sl}_2$-argument (where $\fk {sl}_2$ is generated by $E_{\rho_3}$ and $F_{\rho_3}$), $E_{\rho_3}u$ and $F_{\rho_3}v_{\rho_3}$ are proportional since the (irreducible) $\fk {sl}_2$-subrepresentations generated by $u$ and by $v_{\rho_3}$ agree. Thus, by lemma \ref{lb1}, $E_\alpha u$ and $F_{\rho_4}v_{\rho_4}$ are linearly independent. So $E_\alpha u$ is not in $K_\gk^{\mu_0}(\lambda)[\nu_0-\mu_0]$. This proves (iii). Similarly, in case (4), by lemma \ref{lb5} we have $K_\gk^{\mu_0}(\lambda)[\nu_0-\mu_0]=K_\gk^{\lambda_4}(\lambda_4)[0]=\mbb C\cdot F_{\rho_3}v_{\rho_3}$. Choose $\rho_5=\frac 12[-1,1,1,-1]$ and  non-zero $v_{\rho_5}\in L_\gk(\lambda_4)[\rho_5]$. Then $L_\gk(\lambda)[\eta]=L_\gk(\lambda_4)[\rho_5]$. We have $\alpha=-\rho_5$. So $E_\alpha v_{\rho_5}=E_{-\rho_5}v_{\rho_5}=F_{\rho_5}v_{\rho_5}$, which, by lemma \ref{lb15}, is not in $\mbb C\cdot F_{\rho_3}v_{\rho_3}$. This again proves (iii).
\end{proof}

Thus, we have proved that the intertwining operators of fundamental types are energy-bounded. This proves theorem \ref{lb2}.

\begin{appendices}

\section{Computations by LieART}\label{lb14}

One can calculate the weight multiplicities and the tensor product rules using LiE \cite{LCL92}\footnote{The LiE onlie service can be found on the personal website of M.A.A. van Leuwen. See \url{http://wwwmathlabo.univ-poitiers.fr/~maavl/LiE/form.html}} or the Mathematica package LieART \cite{FK15}. The LieART codes used in this article are provided below.

The following LieART (v.2.0.0) code shows the root system of $\gk=\fk f_4$.
\begin{center}
RootSystem[F4]//OrthogonalBasis
\end{center}
The weights and their multiplicities of $L_\gk(\lambda_4)=L_\gk((0,0,0,1))$ can be calculated by the code:
\begin{center}
WeightSystem[Irrep[F4][0,0,0,1]]//OrthogonalBasis
\end{center}

The following codes are used in the proof of proposition \ref{lb13} to check the first half of condition (b), namely $\dim L_\gk(\nu_0)[\nu-\rho]=1$. The outputs of  these codes are all $1$.\\
$\blt$ Case (1) input:
\begin{center}
WeightMultiplicity[Weight[F4][0,1,-1,1], Irrep[F4][0,0,1,0]]
\end{center}
$\blt$ Case (2) input:
\begin{center}
WeightMultiplicity[Weight[F4][0,1,0,-1], Irrep[F4][0,0,1,0]]
\end{center}
$\blt$ Case (4) input:
\begin{center}
WeightMultiplicity[Weight[F4][0,1,-1,0], Irrep[F4][0,0,0,1]]
\end{center}
$\blt$ Case (5) input:
\begin{center}
WeightMultiplicity[Weight[F4][0,1,-1,1], Irrep[F4][0,0,1,0]]
\end{center}
The code for case (1)  computes the multiplicity of the weight $(0,1,-1,1)$ in $L_\gk((0,0,1,0))=L_\gk(\lambda_3)$. The other codes are understood in a similar way.

To check the second half of condition (b), namely $\Hom_\gk(\nu_0\otimes\rho,\nu)=1$, we calculate   $\nu_0\otimes\rho$, which shows that $\nu$ (boxed in the outputs) appears precisely once in the tensor product. We write $(n_1,n_2,n_3,n_4)$ as $(n_1n_2n_3n_4)$ when none of the four integers exceeds $9$.\\
$\blt$ Case (1) input:
\begin{center}
DecomposeProduct[Irrep[F4][0,0,1,0], Irrep[F4][0,0,1,0]]//StandardForm
\end{center}
Output:
\begin{align*}
&(0010)\otimes (0010)=(0000)+(0001)+(1000)+2(0010)+2(0002)+2(1001)\\
+&(2000)+(0100)+(0003)+2(0011)+(1010)+(1002)+\boxed{(0101)}+(0020)
\end{align*}
$\blt$ Case (2) input:
\begin{center}
DecomposeProduct[Irrep[F4][0,0,1,0], Irrep[F4][0,0,0,1]]//StandardForm
\end{center}
Output:
\begin{align*}
&(0010)\otimes (0001)=(0001)+(1000)+(0010)+(0002)+(1001)+\boxed{(0100)}+(0011)
\end{align*}
$\blt$ Case (4): Same as case (2).\\
$\blt$ Case (5): Same as case (1).

\section{Another proof of lemma \ref{lb1}}\label{lb10}
Due to the importance of lemma \ref{lb1}, we give in this section an alternate proof of this lemma. Let $v\in L_\gk(\lambda_4)[\lambda_4]$ be a highest weight vector with length $1$. Set
\begin{gather*}
\alpha=[1,0,0,-1],\qquad \beta=\frac 12[1,1,1,1]
\end{gather*}
which are roots of $\gk$. Recall $\rho_3=[0,0,0,1],\rho_4=\frac 12[1,-1,-1,-1]$. Then $\alpha+\rho_3=\beta+\rho_4=[1,0,0,0]=\lambda_4$. By scaling $v_{\rho_3},v_{\rho_4}$, we may assume that $v_{\rho_3}=F_\alpha v$ and $v_{\rho_4}=F_\beta v$. We shall show that $F_{\rho_3}F_\alpha v$ and $F_{\rho_4}F_\beta v$ are not parallel by calculating the angle between them.

We first calculate the square length
\begin{align*}
\bk{F_{\rho_3}F_\alpha v|F_{\rho_3}F_\alpha v}=\bk{F_\alpha v|E_{\rho_3}F_{\rho_3}F_\alpha v}.
\end{align*}
Since $[E_{\rho_3},F_{\rho_3}]=H_{\rho_3}$, $E_{\rho_3}v=0$, and since $\rho_3-\alpha=[-1,0,0,2]$ is not a root, the above expression equals
\begin{align*}
&\bk{F_\alpha v|H_{\rho_3}F_\alpha v}=n_{\lambda_4-\alpha,\rho_3}\bk{F_\alpha v|F_\alpha v}=n_{\rho_3,\rho_3}\bk{F_\alpha v|F_\alpha v}\\
=&2\bk{F_\alpha v|F_\alpha v}=2\bk{v|H_\alpha v}=2n_{\lambda_4,\alpha}\lVert v\lVert^2=2.
\end{align*}
A similar calculation shows
\begin{align*}
\bk{F_{\rho_4}F_\beta v|F_{\rho_4}F_\beta v}=2n_{\lambda_4,\beta}\lVert v\lVert^2=2.
\end{align*}
To show that the two vectors are not parallel, we need to show that the absolute value of $\bk{F_{\rho_3}F_\alpha v|F_{\rho_4}F_\beta v}$ is not equal to $2$.

Set $\gamma=\beta-\rho_3=\alpha-\rho_4=\frac 12[1,1,1,-1]$ which is a positive root. We would like to express $[E_{\rho_3},F_\beta]$ in terms of $F_\gamma$. Note that $(F_\gamma|F_\gamma)=2/(\gamma|\gamma)=2$. Since $[F_{\rho_3},F_\beta]=0$ as $\rho_3+\beta=\frac 12[1,1,1,3]$ is not a root, we have 
\begin{align*}
&([E_{\rho_3},F_\beta]|[E_{\rho_3},F_\beta])=(F_\beta|[F_{\rho_3},[E_{\rho_3},F_\beta]])=(F_\beta|[[F_{\rho_3},E_{\rho_3}],F_\beta])\nonumber\\
=&-(F_\beta|[H_{\rho_3},F_\beta])=n_{\beta,\rho_3}(F_\beta|F_\beta)=1\cdot 2=2.
\end{align*}
Thus,  $[E_{\rho_3},F_\beta]=k_1F_\gamma$ for some constant $k_1$ satisfying $|k_1|=1$. Similarly, since $\rho_4+\alpha=\frac 12[3,-1,-1,-3]$ is not a root, and since $(F_\alpha|F_\alpha)=2/(\alpha|\alpha)=1$, we have
\begin{align*}
([E_{\rho_4},F_\alpha]|[E_{\rho_4},F_\alpha])=-(F_\alpha|[H_{\rho_4},F_\alpha])=n_{\alpha,\rho_4}(F_\alpha|F_\alpha)=2\cdot 1=2.
\end{align*}
Thus $[E_{\rho_4},F_\alpha]=k_2F_\gamma$ where $|k_2|=1$. Now, using the fact that $\rho_3-\rho_4=\frac 12[-1,1,1,3]$ is not a root, we  find
\begin{align*}
&\bk{F_{\rho_3}F_\alpha v|F_{\rho_4}F_\beta v}=\bk{F_\alpha v|E_{\rho_3}F_{\rho_4}F_\beta v}=\bk{F_\alpha v|F_{\rho_4}[E_{\rho_3},F_\beta ]v}=k_1\bk{F_\alpha v|F_{\rho_4}F_\gamma v}\\
=&k_1\bk{E_{\rho_4}F_\alpha v|F_\gamma v}=k_1k_2\bk{F_\gamma v|F_\gamma v}=k_1k_2\bk{v|H_\gamma v}=k_1k_2n_{\lambda_4,\gamma}=k_1k_2
\end{align*}
whose absolute value is $1$. This finishes the proof.

\end{appendices}


\noindent {\small \sc Department of Mathematics, Rutgers University, USA.}

\noindent {\em E-mail}: bin.gui@rutgers.edu\qquad binguimath@gmail.com

\begin{thebibliography}{99}
	\footnotesize	
\bibitem[ACL19]{ACL19}
Arakawa, T., Creutzig, T. and Linshaw, A.R., 2019. W-algebras as coset vertex algebras. Inventiones mathematicae, 218(1), pp.145-195.
	
\bibitem[ADL05]{ADL05}
Abe, T., Dong, C.,  Li, H. 2005.  Fusion Rules for the Vertex Operator Algebras $M(1)^+$ and $V_L^+$. Comm. Math. Phys, 253, pp.171-219.	

\bibitem[Ara15a]{Ara15a}
Arakawa, T., 2015. Associated varieties of modules over kac–moody algebras and c 2-cofiniteness of w-algebras. International Mathematics Research Notices, 2015(22), pp.11605-11666.

\bibitem[Ara15b]{Ara15b}
Arakawa, T., 2015. Rationality of W-algebras: principal nilpotent cases. Annals of Mathematics, pp.565-604.

\bibitem[CKLW18]{CKLW18}
Carpi, S., Kawahigashi, Y., Longo, R. and Weiner, M., 2018. From vertex operator algebras to conformal nets and back (Vol. 254, No. 1213). Memoirs of the American Mathematical Society
	
\bibitem[DLM97]{DLM97}
Dong, C., Li, H. and Mason, G., 1997. Regularity of Rational Vertex Operator Algebras. Advances in Mathematics, 132(1), pp.148-166.

\bibitem[DL14]{DL14}
Dong, C. and Lin, X., 2014. Unitary vertex operator algebras. Journal of algebra, 397, pp.252-277.

\bibitem[FK15]{FK15}
Feger, R. and Kephart, T.W., 2015. LieART—a Mathematica application for Lie algebras and representation theory. Computer Physics Communications, 192, pp.166-195.

\bibitem[Fuc94]{Fuc94}
Fuchs, J., 1994. Fusion rules in conformal field theory. Fortschritte der Physik/Progress of Physics, 42(1), pp.1-48.

\bibitem[FZ92]{FZ92}
Frenkel, I.B. and Zhu, Y., 1992. Vertex operator algebras associated to representations of affine and Virasoro algebras. Duke Mathematical Journal, 66(1), pp.123-168.


	\bibitem[Gui19a]{Gui19a}
	Gui, B., 2019. Unitarity of the modular tensor categories associated to unitary vertex operator algebras, I,  Comm. Math. Phys., 366(1), pp.333-396. 
	

	
	\bibitem[Gui19b]{Gui19b}
	Gui, B., 2019. Energy bounds condition for intertwining operators of type $ B $, $ C $, and $ G_2 $ unitary affine vertex operator algebras. Trans. Amer. Math. Soc. 372 (2019), 7371-7424
	

	
\bibitem[Gui20a]{Gui20a}
	Gui, B., 2020. Unbounded field operators in categorical extensions of conformal nets. arXiv preprint arXiv: 2001.03095.
	
	\bibitem[Gui20b]{Gui20b}
	Gui, B., 2020. Regular vertex operator subalgebras and compressions of intertwining operators. arXiv: 2003.02921.
	
\bibitem[Hua95]{Hua95}
Huang, Y.Z., 1995. A theory of tensor products for module categories for a vertex operator algebra, IV. Journal of Pure and Applied Algebra, 100(1-3), pp.173-216.

\bibitem[Hua05]{Hua05}
Huang, Y.Z., 2005. Differential equations and intertwining operators. Communications in Contemporary Mathematics, 7(03), pp.375-400.	
	
\bibitem[KM15]{KM15}
Krauel, M. and Miyamoto, M., 2015. A modular invariance property of multivariable trace functions for regular vertex operator algebras. Journal of Algebra, 444, pp.124-142.	
	

	
\bibitem[LCL92]{LCL92}
M. A. A. van Leeuwen, A. M. Cohen and B. Lisser, "LiE, A Package for Lie Group Computations", Computer Algebra Nederland, Amsterdam, ISBN 90-74116-02-7, 1992
	
\bibitem[Loke04]{Loke94}
Loke, T.M., 1994. Operator algebras and conformal field theory of the discrete series representations of $\mathrm{Diff}(S^1)$ (Doctoral dissertation, University of Cambridge).

\bibitem[NT05]{NT05}
Nagatomo, K. and Tsuchiya, A., 2005. Conformal field theories associated to regular chiral vertex operator algebras, I: Theories over the projective line. Duke Mathematical Journal, 128(3), pp.393-471.

\bibitem[TUY89]{TUY89}
Tsuchiya, A., Ueno, K. and Yamada, Y., 1989. Conformal field theory on universal family of stable curves with gauge symmetries. In Integrable Sys Quantum Field Theory (pp. 459-566). Academic Press.

\bibitem[Ten19]{Ten19}
	Tener, J.E., 2019. Representation theory in chiral conformal field theory: from fields to observables. Selecta Mathematica, 25(5), p.76.
	
\bibitem[TL04]{TL04}
Toledano-Laredo, V., 2004. Fusion of positive energy representations of lspin (2n). arXiv preprint math/0409044.
	
\bibitem[Was98]{Was98}
Wassermann, A., 1998. Operator algebras and conformal field theory III. Fusion of positive energy representations of $LSU(N)$ using bounded operators. arXiv preprint math/9806031.


\end{thebibliography}
\end{document}